\documentclass[11pt, a4paper, oneside]{amsart}
\usepackage{amsrefs,mathrsfs,amssymb}
\usepackage[T1]{fontenc}
\usepackage[utf8]{inputenc}
\usepackage[body={16cm, 24cm}]{geometry}
\usepackage{graphicx,mathrsfs,amssymb,color}
\usepackage[misc]{ifsym} 

\newcounter{results}[section]

\theoremstyle{plain}
\newtheorem{theorem}[results]{Theorem}
\newtheorem{lemma}[results]{Lemma}
\newtheorem{proposition}[results]{Proposition}
\newtheorem{corollary}[results]{Corollary}

\theoremstyle{remark}
\newtheorem{remark}[results]{Remark}

\theoremstyle{definition}
\newtheorem{definition}[results]{Definition}
\newtheorem{hypo}[results]{Hypotesis}

\numberwithin{equation}{section}

\newcommand{\R}{\ensuremath{\mathbb R}} 
\newcommand{\N}{\ensuremath{\mathbb N}} 
\newcommand{\weakto}{\ensuremath{\rightharpoonup}} 

\newcommand{\eps}{\ensuremath{\varepsilon}} 
\renewcommand{\div}{\mathrm{div}\,} 

\DeclareMathOperator*{\argmin}{arg\,min} 

\renewcommand{\d}{\,\mathrm d}  

\newcommand{\prob}{\ensuremath{\mathcal{P}}} 
\newcommand{\m}[1]{\mathsf m_2(#1)} 
\newcommand{\meps}[1]{\mathsf m_{2+\eps}(#1)} 
\newcommand{\mm}{\mathsf m_2}
\newcommand{\ev}{\mathsf e} 
\newcommand{\adm}[2]{\mathcal A_{#1}(#2)} 
\newcommand{\F}{\mathrm{F}} 
\newcommand{\V}{\mathrm{V}} 
\renewcommand{\H}{\mathcal{H}} 
\newcommand{\HF}{\mathcal{H}_{\F}} 
\renewcommand{\S}{\mathrm{S}} 
\renewcommand{\SS}{\mathscr{S}} 
\newcommand{\eeta}{\boldsymbol{\eta}}

\newcommand{\ppi}{\boldsymbol{\pi}}

\newcommand{\ssigma}{\boldsymbol{\sigma}}
\newcommand{\sfs}{\mathsf s} 

\newcommand{\rmC}{\mathrm C} 
\newcommand{\epi}[1]{\mathrm{epi}\,(#1)} 
\newcommand{\tto}{\rightrightarrows} 

\allowdisplaybreaks

\title[Reachability for multiagent control systems via Lyapunov functions]{Reachability for multiagent control systems via Lyapunov functions  }

\author{Giulia Cavagnari (\Letter) }
\address{Giulia Cavagnari: Politecnico di Milano, Dipartimento di Matematica, Piazza Leonardo Da Vinci 32, 20133 Milano (Italy)}
\email{giulia.cavagnari@polimi.it}

\author{Marc Quincampoix}
\address{Marc Quincampoix: UMR CNRS 6205, Univ Brest,
  Laboratoire de Math\'ematiques de Bretagne Atlantique, 6 avenue Victor Le Gorgeu, 29200 Brest (France)}
\email{marc.quincampoix@univ-brest.fr}

 \thanks{\emph{Acknowledgments.} This research is supported by INdAM-GNAMPA under the Project 2024 ``Controllo ottimo infinito dimensionale: aspetti teorici ed applicazioni'', CUP E53C23001670001.}

\begin{document}

\begin{abstract} This paper concerns the problem of reachability of a given state for a  multiagent control system in $\R ^d$. In such a system, at every time each agent can choose his/her velocity which depends both on his/her position and on the position of the whole crowd of agents (modeled by a probability measure on $ \R ^d$). The main contribution of the paper is to study the above reachability problem with a given rate of attainability through a Lyapunov method adapted to the Wasserstein space of probability measures. As a byproduct we obtain a new comparison result for viscosity solutions of Hamilton Jacobi equations in the Wasserstein space.

\end{abstract}

\maketitle


\thispagestyle{empty}


\section{Introduction} 
During the last years, there has been a large interest on  controlled multiagent systems modeled on the Wasserstein space of probabilty measures on $ \R ^d$. In such a system, the number of agents  is so huge that only a statistical description is possible. Namely the state variable $ \mu $  is a probability measure on $ \R ^d$ such that for any (Borel) set $A \subset \R ^d$,  $ \mu (A)$ is the fraction of the total amount of agents that are present in $A$. The time evolution $ t \mapsto \mu _t$ of the state variable occurs according to a two-level dynamics, as follows.

{\em - Microscopic level}. Each agent at position $x(t)$ at time $t$ follows 
$$ \dot x (t) \in \F (x (t) , \mu _t) \mbox{  for almost every $t \geq 0$,  } $$
where $ \F : \R ^d \times \prob_2(\R^d) \tto \R^d$ is a set-valued map and $\prob _2(\R^d)$ is the set of probability measures on $\R^d$ with finite second moment.
Namely, the velocity of each agent depends both  on his own position and on the position of the whole crowd of agents. It is worth pointing out that we make an indistiguishability assumption meaning that an agent does not interact individually with every other agents but only with the whole crowd of agents.

{\em - Macroscopic level}. The dynamics of the crowd of agents is given by the following so-called  {\it continuity equation} 
\begin{equation}\label{eqmu}
\partial_t\mu_t+\mathrm{div}(v_t\mu_t)=0 \mbox{ for a.e. $t$ and $\mu_t$-a.e. $x\in \mathbb R^d $, } 
\end{equation} where $ v_t (\cdot)$ is a Borel time dependent vector field on $ \R ^d$.
This equation holds true  in the sense of distributions on $ \R \times\R^d$ (cf. Definition \ref{def-macro} later on). 

The coupling between microscopic and macroscopic dynamics is given by 
\begin{equation}\label{coupling}
 v_t (x) \in \F(x,\mu_t)  \mbox{ for a.e. $t$ and $\mu_t$-a.e. $x\in \mathbb R^d $.} 
\end{equation}
\medskip

Such bi-level dynamics have been widely studied in the literature. For the continuity equation we refer to \cite{ags}, the control dynamics was studied in \cite{CMR,CMQ, CLOS, CMP, JMQ20, JMQ21} (cf. also \cite{AMQ, BF22, BFarxiv, BadF22}).

\medskip

Our main goal is to study a reachability problem: the controller aims to drive the system to a given probability measure $\bar \mu \in \prob _ 2(\R^d)$ following an admissible solution of the dynamical system \eqref{eqmu},\eqref{coupling}. Some results in this direction have been obtained in \cite{DMR}, where the authors study the controllability problem in the case where the Lipschitz control is localized to act in a prescribed region. However, in addition to the required regularity, the authors address the problem in the case of local dynamics, thus without interactions.

Analogously to the classical case \cite{Au}, we aim to provide sufficient conditions for reachability using a Lyapunov function approach under more general assumptions on the dynamics, including non-locality. A recent contribution for Measure Differential Equations in this direction is given in \cite{DMPV}. We assume that there exists a Lipschitz continuous function $V :  \prob _ 2(\R^d) \to \R _+$ such that
$$ \{ \bar \mu \} = \{ \mu \in  \prob _ 2(\R^d) \,:\, V ( \mu ) =0\} .$$
It is well known that the requirement for $V$ to serve as a Lyapunov function is closely linked to the Hamiltonian of the given dynamical system.
Of course, since the Wasserstein space $\prob _2(\R^d)$  is not a vector space and the Lyapunov functions could not  be regular enough, we need a suitable notion of solution for stating the corresponding partial differential inequality satisfied by the Lyapunov function. We will use the notion of viscosity (super) solution in the Wasserstein space adopted in \cite{JMQ20, JMQ21}. Various notions of viscosity solution for Hamilton Jacobi equations have been investigated in \cite{AJZ, CQ, G1, G2, Jimenez, JMQ20, JMQ21, MQ}.

The main results of the paper are contained in Theorem \ref{thm:main}. {\color{black} In particular, we investigate an asymptotic rate of convergence of the trajectory with an exponential decay.} A key ingredient in our approach is a {\color{black}comparison result for Hamilton Jacobi equations in Wasserstein spaces, which is a byproduct of our article. Unlike the analogous result developed in \cite{JMQ20}, which serves as inspiration for our approach, we do not assume boundedness of the viscosity solutions involved. Instead, we show that the result can be established under a suitable Lipschitz continuity condition, which fits more naturally within our framework.}

\medskip

The paper is organized as follows: a preliminary section is devoted to notations, essential recalls on tools of optimal transport theory and introduction of the dynamical setting under investigation. In Section \ref{sec:HJB}, we recall basic facts on Hamilton Jacobi Equations and we state an adapted version of the comparison principle for viscosity solutions on the Wasserstein space, whose proof is postponed to Appendix \ref{sec:appA}. Section \ref{sec:reach} contains the core of the paper{\color{black}; examples of application of the theory are presented in Section \ref{sec:examples}}. Finally, Appendix \ref{app0} is devoted to some technical proofs.

%
\section{Preliminaries and settings} 
In this section, we first recall standard definitions and results of measure theory and optimal transport theory for which we refer to the survey \cite{ags}.
Then, we introduce the dynamical setting under consideration.

\subsection{The Wasserstein metric}\label{sec:W2}
Let $X$ be a separable Banach space (here $X$ will be either $\R^d$ or $\rmC([a,b];\R^d)$, the space of $\R^d$-valued continuous curves on $[a,b]\subset\R$ endowed with the uniform norm). 
We denote by $\prob(X)$ the space of Borel probability measures on $X$, endowed with the topology of narrow convergence\footnote{ We recall  that a sequence $\{\nu_n\}_{n\in\N}\subset\prob(X)$ \emph{narrowly converges} to $\nu\in\prob(X)$, and we write $\nu_n\weakto\nu$ as $n\to+\infty$, if for any $f\in\rmC([a,b];\R)$ bounded
$ \int_X f(x)\d\nu_n(x)\to\int_X f(x)\d\nu(x),\quad\text{as }n\to+\infty .$
}.

Given $Y$ separable Banach space, $r:X\to Y$ a Borel measurable map and $\nu\in\prob(X)$, we denote by $r_\sharp\nu\in\prob(Y)$ the \emph{push-forward} measure defined by
\[\int_Y f(y)\d(r_\sharp\nu)(y):=\int_Xf(r(x))\d\nu(x),\]
for any $f:Y\to[0,+\infty)$ Borel measurable function.
Given $\nu\in\prob(X)$, we denote by $ \m{\nu}$  its \emph{$2$-moment}
\[\m{\nu}:=\left(\int_X |x|^2\d\nu(x)\right)^{\frac{1}{2}}\]
and  by $\prob_2(X)$ the set of  $\nu\in\prob(X)$ such that $\m{\nu}<+\infty$. On this space, we consider the Wasserstein metric $W_2$ defined by
\begin{equation}\label{eq:W2}
W_2(\nu_1,\nu_2):=\left(\inf_{\ssigma\in\Gamma(\nu_1,\nu_2)}\int_{X\times X}|x_1-x_2|^2\d\ssigma(x_1,x_2)\right)^{1/2},
\end{equation}
where $\Gamma(\nu_1,\nu_2)$ denotes the set of admissible couplings/plans between $\nu_1$ and $\nu_2$, given by
\[\Gamma(\nu_1,\nu_2):=\left\{\ssigma\in\prob(X\times X)\,:\,\pi^1_\sharp\ssigma=\nu_1,\,\pi^2_\sharp\ssigma=\nu_2\right\},\]
and $\pi^i:X\times X\to X$, $i=1,2$, are the projections maps defined by
$\pi^i(x_1,x_2)=x_i,\quad i=1,2.$
The set of optimal couplings realizing the infimum in \eqref{eq:W2} is denoted by $\Gamma_o(\nu_1,\nu_2)$. We recall that the topology induced by the $W_2$ metric on $\prob_2(X)$ is stronger than the narrow topology. (cf. e.g. \cite[Remark 7.1.11]{ags}).

We also recall that a map $T:\R^d\to\R^d$ is an \emph{optimal transport map} between $\nu _1 \in\prob_2(\R^d)$ and $\nu _2\in\prob_2(\R^d)$ if 
  $\nu _2 = T_\sharp \nu _1$  and the induced plan is optimal, i.e. $(\mathrm{id}_{\R^d},T)_\sharp \nu _1 \in\Gamma_o (\nu _1 , \nu _2)$.

\subsection{Trajectories of the control system }\label{sec:adm}
We introduce the controlled non-local dynamics under consideration in the form of a differential inclusion. In particular, in this section we provide assumptions on the involved set-valued map $\F:\R^d\times\prob_2(\R^d)\tto\R^d$ that will be required later on, we give the definition of admissible trajectories in the Wasserstein space $(\prob_2(\R^d),W_2)$ and of associated Hamiltonian. This is the setting considered in \cite{JMQ20} (cf. also \cite{JMQ21,CMQ,CMP,CLOS}).

\medskip

Denote by $\F:\R^d\times\prob_2(\R^d)\tto\R^d$ the set-valued vector field involved in the controlled dynamics under study. Whenever specified, we work under the following assumptions, where on the space $\R^d\times\prob_2(\R^d)$ we consider the metric
\[d\left((x_1,\nu_1),(x_2,\nu_2)\right):=|x_1-x_2|+W_2(\nu_1,\nu_2),\qquad (x_i,\nu_i)\in\R^d\times\prob_2(\R^d),\,i=1,2.\]
\begin{hypo}\label{hypo}\
\begin{enumerate}
\item\label{itemh1} $\F$ has convex, compact and nonempty images;
\item\label{item:h2} $\F$ is $L$-Lipschitz continuous, i.e. there exists $L>0$ such that
\[\F(x_1,\nu_1)\subset \F(x_2,\nu_2)+L\,d\left((x_1,\nu_1),(x_2,\nu_2)\right)\,\overline{B(0,1)},\]
for any $(x_i,\nu_i)\in\R^d\times\prob_2(\R^d)$, $i=1,2$.
\end{enumerate}
\end{hypo}

Denoted by
\begin{equation}\label{eq:KF}
K_\F:=\max_{v\in\F(0,\delta_0)}|v|,
\end{equation}
thanks to \eqref{item:h2}, we get the following  estimate 
\begin{equation}\label{eq:growthF}
\F(x,\nu)\subset \F(0,\delta_0)+L\left( \m{\nu}+|x|\right)\,\overline{B(0,1)}\subset
C(1+\m{\nu})\,(1+|x|) \,\overline{B(0,1)},
\end{equation}
for any $(x,\nu)\in\R^d\times\prob_2(\R^d)$, where {\color{black}$C:=\max\{L,K_\F\}$}.

We use the notation $\mu=(\mu_t)_{t\in I}$ to denote a curve $\mu:I\to\prob_2(\R^d)$, where $I$ is a given interval.
\begin{definition}[Admissible trajectories]\label{def-macro}
Let $I=[a,b]\subset\R$, $\bar\mu\in\prob_2(\R^d)$ and $\mu=(\mu_t)_{t\in I}\subset \prob_2(\R^d)$ be an absolutely continuous curve. We say that $\mu$ is an \emph{admissible trajectory starting from $\bar\mu$}, and we write $\mu\in\adm{I}{\bar{\mu}}$, if there exists a Borel measurable vector field $v:I\times\R^d\to\R^d$, $v_t\in L^2_{\mu_t}(\R^d;\R^d)$ for a.e. $t\in I$, satisfying
\begin{itemize}
\item $v_t(x)\in \F(x,\mu_t)$ for a.e. $t\in I$ and $\mu_t$-a.e. $x\in \mathbb R^d$;
\item $\partial_t\mu_t+\mathrm{div}(v_t\mu_t)=0$ in the sense of distributions on $I\times\R^d$, equivalently 
\[\dfrac{d}{dt}\int_{\R^d}\phi(x)\d\mu_t(x)=\int_{\R^d}\langle\nabla\phi(x),v_t(x)\rangle\d\mu_t(x),\]
for a.e. $t\in I$ and all $\phi\in C^1_c(\R^d)$.
\end{itemize}
We say that a locally absolutely continuous curve $\mu=(\mu_t)_{t\in[0,+\infty)}$ is admissible starting from $\bar{\mu}$, and we write $\mu\in\adm{[0,+\infty)}{\bar{\mu}}$, if the restriction $\mu|_{[0,T]}$ of $\mu$ to the compact interval $[0,T]$ is admissible starting from $\bar{\mu}$ for any $T>0$.
\end{definition}

We postpone to Appendix \ref{app0} the computation of standard a priori estimates on the $2$-moments along admissible trajectories.

\section{Hamilton Jacobi Bellman equation}\label{sec:HJB}
The definition of Lyapunov function requires a suitable notion of variational inequality related to Hamilton-Jacobi equation.  Such equation is solved by the  value function of the following  Mayer-type optimal control problem \cite{JMQ20}.
 Given a Lipschitz continuous  final cost $g:\prob_2(\R^d)\to\R$, we define the value function $U_g:[a,b]\times\prob_2(\R^d)\to\R$ by
\[U_g(t,\nu):=\inf\left\{g(\mu_{b})\,:\,\mu\in\adm{[t,b]}{\nu}\right\}.\]
Under Hypothesis \ref{hypo}, it is possible to prove the existence of optimal trajectories (cf. \cite[Corollary 2]{JMQ20}) and a Dynamic Programming Principle (cf. \cite[Proposition 3]{JMQ20}), which states in particular that the map
\[t\mapsto U_g(t,\mu_t)\]
is non decreasing in $[a,b]$ for any $\mu=(\mu_t)_{t\in[a,b]}\in\adm{[a,b]}{\bar{\mu}}$, and it is constant along optimal trajectories.

\medskip
We introduce  the following \emph{Hamilton-Jacobi-Bellman} equation in $[a,b]\times\prob_2(\R^d)$
(with  $a<b$), 
\begin{equation}\label{eq:HJBintro}
\partial_t U(t,\nu)+\H(\nu,D_\nu U(t,\nu))=0,
\end{equation}
where $\H(\nu,p)$ is defined for any $\nu\in\prob_2(\R^d)$, $p\in L^2_\nu(\R^d)$.
Since $U:[a,b]\times\prob_2(\R^d)\to\R$ may not be regular enough, equation \eqref{eq:HJBintro} has to be interpreted in a viscosity sense which we will specify later on in Section \ref{sec:visco}.

In the case of the Mayer problem, we will consider  the following  Hamiltonian associated to the set-valued map $\F$
\begin{equation}\label{eq:H}
\HF(\nu,p):=\int_{\R^d}\inf_{v\in\F(x,\nu)}\langle p(x),v\rangle \d\nu(x),\qquad p\in L^2_\nu(\R^d),\,\nu\in\prob_2(\R^d).
\end{equation}

\subsection{Sub/super-differentials}\label{sec:diff}
	
Before giving the definition of sub/super-differential in $\R^d\times\prob_2(\R^d)$ or $\prob_2(\R^d)$, taken from \cite{JMQ20,JMQ21} (cf. also \cite{ags, Jimenez}), we need to recall the following.
\begin{definition}[Optimal displacement]\label{def:optdis}
Let $\bar{\nu}\in\prob_2(\R^d)$. A function $p\in L^2_{\bar{\nu}}(\R^d)$ is called \emph{optimal displacement from $\bar{\nu}$} if there exists $\lambda>0$ and an optimal transport map $T:\R^d\to\R^d$ between $\bar{\nu}$ and $T_\sharp\bar{\nu}$ such that
\[p=\lambda(T-\mathrm{id}_{\R^d}).\]
\end{definition}
\begin{remark}
We recall the following equivalent characterization of optimal displacement (\cite[Lemma 4]{JMQ20}): $p$ is an \emph{optimal displacement from $\bar{\nu}\in\prob_2(\R^d)$} if there exist $\lambda>0$, $\xi\in\prob_2(\R^d)$ and $\ssigma\in\Gamma_o(\bar{\nu},\xi)$ such that
\[p(x)=\lambda\left(\int_{\R^d} y\d\sigma_x(y)-x\right),\]
where $\{\sigma_x\}_{x\in\R^d}\subset\prob_2(\R^d)$ is the family obtained by disintegrating $\ssigma$ w.r.t. the projection on the first marginal $\bar\nu$ (cf. \cite[Theorem 5.3.1]{ags}).
\end{remark}

\begin{definition}[Sub/super-differential in $\R\times\prob_2(\R^d)$]\label{def:diff}
Let $a,b\in\R$, $a<b$, and consider the map $U:[a,b]\times\prob_2(\R^d)\to\R$. Given $(\bar{t},\bar{\nu})\in[a,b]\times\prob_2(\R^d)$ and $\delta\ge0$, we say that the pair $(p_{\bar{t}},p_{\bar{\nu}})\in\R\times L^2_{\bar{\nu}}(\R^d)$ is a \emph{(total) $\delta$-viscosity subdifferential} of $U$ at $(\bar{t},\bar{\nu})$, and we write $(p_{\bar{t}},p_{\bar{\nu}})\in\boldsymbol{D}^-_\delta U(\bar{t},\bar{\nu})$, if
\begin{enumerate}
\item $p_{\bar{\nu}}$ is an optimal displacement from $\bar{\nu}$;
\item for all $(t,\nu)\in[a,b]\times\prob_2(\R^d)$ and $\ssigma\in\Gamma(\bar{\nu},\nu)$,
\begin{equation*}
U(t,\nu)-U(\bar{t},\bar{\nu})\ge p_{\bar{t}}\,(t-\bar{t})+\int_{\R^d\times\R^d}\langle p_{\bar{\nu}}(x),y-x\rangle \d\ssigma(x,y)
-\delta\cdot\Delta_{t,\ssigma}+o\left(\Delta_{t,\ssigma}\right),
\end{equation*}
where $\Delta_{t,\ssigma}:=\sqrt{|t-\bar{t}|^2+\int |x-y|^2\d\ssigma(x,y)}$.
\end{enumerate}
Analogously, the set of \emph{(total) $\delta$-viscosity superdifferentials} of $U$ at $(\bar{t},\bar{\nu})$ is given by $\boldsymbol{D}^+_\delta U(\bar{t},\bar{\nu}):=-\boldsymbol{D}^-_\delta (-U)(\bar{t},\bar{\nu})$.
\end{definition}

If the map $U$ does not depend on $t$ or it is constant w.r.t. $t$ the above definition becomes
\begin{definition}[Sub/super-differential in $\prob_2(\R^d)$]\label{def:diffP2}
Let $U:\prob_2(\R^d)\to\R$. Given $\bar{\nu}\in\prob_2(\R^d)$ and $\delta\ge0$, we say that $p\in L^2_{\bar{\nu}}(\R^d)$ is a \emph{$\delta$-viscosity subdifferential} of $U$ at $\bar{\nu}$, and we write $p \in\partial^-_\delta U(\bar{\nu})$, if
\begin{enumerate}
\item $p$ is an optimal displacement from $\bar{\nu}$;
\item for all $\nu\in\prob_2(\R^d)$ and $\ssigma\in\Gamma(\bar{\nu},\nu)$,
\begin{equation*}
U(\nu)-U(\bar{\nu})\ge \int_{\R^d\times\R^d}\langle p(x),y-x\rangle \d\ssigma(x,y)
-\delta\cdot\Delta_{\ssigma}+o\left(\Delta_{\ssigma}\right),
\end{equation*}
where $\Delta_{\ssigma}:=\left(\int |x-y|^2\d\ssigma(x,y)\right)^{1/2}$.
\end{enumerate}
Analogously, the set of \emph{$\delta$-viscosity superdifferentials} of $U$ at $\bar{\nu}$ is given by $\partial^+_\delta U(\bar{\nu}):=-\partial^-_\delta (-U)(\bar{\nu})$.
\end{definition}

In case $U:[a,b]\times\prob_2(\R^d)\to\R$ is partially differentiable w.r.t. $t$, in the following proposition we prove a relation between the pair composed by the (partial) derivative in $(a,b)$ and the (partial) $\delta$-differential in $\prob_2(\R^d)$, with the pair of total $\delta$-differential in $(a,b)\times\prob_2(\R^d)$.
\begin{proposition}\label{prop:incldiff2.0}
Let $a,b\in\R$, $a<b$, $(\bar{t},\bar{\nu})\in(a,b)\times\prob_2(\R^d)$ and $\delta\ge0$. Let $U:[a,b]\times\prob_2(\R^d)\to\R$ be partially differentiable w.r.t. $t\in(a,b)$ at $(\bar{t},\bar{\nu})$, so that there exists $\partial_t U(\bar{t},\bar{\nu})$. Then
\begin{equation}\label{eq:incldiff2-}
(p_{\bar{t}},p_{\bar{\nu}})\in\boldsymbol{D}^-_\delta U(\bar{t},\bar{\nu}) \Rightarrow
\begin{cases}
p_{\bar{t}}\in \overline{B\left(\partial_t U(\bar{t},\bar{\nu}),\delta\right)},\\
p_{\bar{\nu}}\in\partial^-_\delta U(\bar{t},\bar{\nu}).
\end{cases}
\end{equation}
Analogously, we have
\begin{equation}\label{eq:incldiff2+}
(p_{\bar{t}},p_{\bar{\nu}})\in\boldsymbol{D}^+_\delta U(\bar{t},\bar{\nu}) \Rightarrow
\begin{cases}
p_{\bar{t}}\in \overline{B\left(\partial_t U(\bar{t},\bar{\nu}),\delta\right)},\\
p_{\bar{\nu}}\in\partial^+_\delta U(\bar{t},\bar{\nu}).
\end{cases}
\end{equation}
\end{proposition}
\begin{proof}
We prove \eqref{eq:incldiff2-} since the proof of \eqref{eq:incldiff2+} is analogous.
Let $(p_{\bar{t}},p_{\bar{\nu}})\in\boldsymbol{D}^-_\delta U(\bar{t},\bar{\nu})$. Then, by Definition \ref{def:diff}, $p_{\bar{\nu}}$ is an optimal displacement from $\bar{\nu}$ and for any $(t,\nu)\in[a,b]\times\prob_2(\R^d)$ and $\ssigma\in\Gamma(\bar{\nu},\nu)$, we have
\begin{equation}\label{eq:proofincldiff2-}
U(t,\nu)-U(\bar{t},\bar{\nu})\ge p_{\bar{t}}\,(t-\bar{t})+\int_{\R^d\times\R^d}\langle p_{\bar{\nu}}(x),y-x\rangle \d\ssigma(x,y)
-\delta\cdot\Delta_{t,\ssigma}+o\left(\Delta_{t,\ssigma}\right).
\end{equation}

We first prove that $p_{\bar{t}}\in \overline{B\left(\partial_t U(\bar{t},\bar{\nu}),\delta\right)}$: by choosing $\nu=\bar{\nu}$ and $\ssigma=\left(\mathrm{id},\mathrm{id}\right)_\sharp\bar{\nu}$ in \eqref{eq:proofincldiff2-}, we get
\[U(t,\bar{\nu})-U(\bar{t},\bar{\nu})\ge p_{\bar{t}}\,(t-\bar{t})
-\delta |t-\bar{t}|+o\left(|t-\bar{t}|\right),\]
for any $t\in[a,b]$.
By partial differentiability of $U$ w.r.t. $t$, we can expand by Taylor's formula the left-hand side and get
\[\partial_t U(\bar{t},\bar{\nu})\cdot(t-\bar{t})+o\left(|t-\bar{t}|\right)\ge p_{\bar{t}}\,(t-\bar{t})
-\delta |t-\bar{t}|+o\left(|t-\bar{t}|\right),\]
for any $t\in[a,b]$. Thus, dividing by $|t-\bar{t}|$ and letting $|t-\bar{t}|\to0$, we finally obtain
\[\partial_t U(\bar{t},\bar{\nu})-\delta\le p_{\bar{t}}\le \partial_t U(\bar{t},\bar{\nu})+\delta.\]

It remains to prove that $p_{\bar{\nu}}\in\partial^-_\delta U(\bar{t},\bar{\nu})$. By choosing $t=\bar{t}$ in \eqref{eq:proofincldiff2-}, we have
\[U(\bar{t},\nu)-U(\bar{t},\bar{\nu})\ge \int_{\R^d}\langle p_{\bar{\nu}}(x),y-x\rangle \d\ssigma(x,y)
-\delta\cdot\Delta_{\ssigma}+o\left(\Delta_{\ssigma}\right),\]
for any $\nu\in\prob_2(\R^d)$ and $\ssigma\in\Gamma(\bar{\nu},\nu)$.
By Definition \ref{def:diffP2}, we conclude.
\end{proof}

\begin{remark}\label{rmk:problem1}
In particular, if $\delta=0$, by Proposition \ref{prop:incldiff2.0}, we have
\begin{equation*}
\boldsymbol{D}^-_0 U(\bar{t},\bar{\nu})\subset\left\{(\partial_t U(\bar{t},\bar{\nu}),p)\,:\,p\in\partial^-_0 U(\bar{t},\bar{\nu})\right\},
\end{equation*}
and the inclusion may be strict.
\end{remark}

\subsection{Viscosity solutions and the Comparison principle}\label{sec:visco}

Using the above notion of sub/super differential, we are able to give a meaning of viscosity solution to the Hamilton Jacobi Bellman equation \eqref{eq:HJBintro}.

\begin{definition}[Viscosity solutions]\label{def:viscosol}
A function $U:[a,b]\times\prob_2(\R^d)\to\R$ is
\begin{enumerate}
\item a \emph{subsolution} of \eqref{eq:HJBintro} if $U$ is u.s.c. and there exists a map $C:\prob_2(\R^d)\to(0,+\infty)$, $C(\cdot)$ bounded on bounded sets, such that
\[p_t+\H(\nu,p_\nu)\ge -C(\nu)\delta,\]
for all $(t,\nu)\in[a,b]\times\prob_2(\R^d)$, $(p_t,p_\nu)\in\boldsymbol{D}^+_\delta U(t,\nu)$ and $\delta>0$;
\item a \emph{supersolution} of \eqref{eq:HJBintro} if $U$ is l.s.c. and there exists a map $C:\prob_2(\R^d)\to(0,+\infty)$, $C(\cdot)$ bounded on bounded sets, such that
\[p_t+\H(\nu,p_\nu)\le C(\nu)\delta,\]
for all $(t,\nu)\in[a,b]\times\prob_2(\R^d)$, $(p_t,p_\nu)\in\boldsymbol{D}^-_\delta U(t,\nu)$ and $\delta>0$;
\item a \emph{solution} of  \eqref{eq:HJBintro} if $U$ is both a supersolution and a subsolution.
\end{enumerate}
\end{definition}

We state here {\color{black}in Theorem \ref{thm:comparison}} the Comparison Principle for \eqref{eq:HJBintro} that will be used in this paper {\color{black} and whose proof is postponed to Appendix \ref{sec:appA}. 
The result involves a suitable Lipschitz continuity notion that we call $\mm$-local Lipschitzianity.}
\begin{definition}[{\color{black}$\mm$-local Lipschitzianity}]\label{def:locLip}
Consider a map $\phi:\prob_2(\R^d)\to\R$. We say that 
\begin{enumerate}
\item\label{item:m2lip} $\phi$ is \emph{$\mm$-locally Lipschitz} if for any $R>0$ there exists $L_R>0$ such that
\begin{equation}\label{eq:lipphi}
|\phi(\nu_1)-\phi(\nu_2)|\le L_R\, W_2(\nu_1,\nu_2),
\end{equation}
for any $\nu_i\in\prob_2(\R^d)$ such that $\m{\nu_i}\le R$, $i=1,2$;
\item\label{item:loclip}  $\phi$ is \emph{locally Lipschitz continuous} if for any $\nu_1\in\prob_2(\R^d)$ there exist $R_{\nu_1}>0$ and $L_{\nu_1}>0$ such that 
\[|\phi(\nu_1)-\phi(\nu_2)|\le L_{\nu_1}\, W_2(\nu_1,\nu_2),\]
for any $\nu_2\in\prob_2(\R^d)$ such that $W_2(\nu_1,\nu_2)\le R_{\nu_1}$;
\item $\phi$ is \emph{Lipschitz continuous} if there exists $L>0$ such that
\[|\phi(\nu_1)-\phi(\nu_2)|\le L\, W_2(\nu_1,\nu_2),\]
 for any $\nu_i\in\prob_2(\R^d)$, $i=1,2$.
 \end{enumerate}
\end{definition}

\begin{remark}\label{rmk:m2}
Notice that, according to Definition \ref{def:locLip}, if $\phi$ satisfies either \eqref{item:m2lip} or \eqref{item:loclip}, then $\phi$ is continuous in the $W_2$-metric. Furthermore, if $\phi:\prob_2(\R^d)\to\R$ is Lipschitz continuous, then it is trivially $\mm$-locally Lipschitz. The viceversa does not hold. As an example, the function $\tilde\phi:\prob_2(\R^d)\to\R$, defined by
$\tilde\phi(\nu):=\frac{1}{2}W_2^2(\nu,\delta_0)=\frac{1}{2}\mm^2(\nu),$
is not Lipschitz continuous, however it is $\mm$-locally Lipschitz. Indeed, by triangular inequality we have for any $\nu_i\in\prob_2(\R^d)$, $i=1,2$,
\begin{align*}
|\tilde\phi(\nu_1)-\tilde\phi(\nu_2)|&\le \frac{1}{2}\left|\m{\nu_1}-\m{\nu_2}\right|\left(\m{\nu_1}+\m{\nu_2}\right)\\
&\le\frac{1}{2}W_2(\nu_1,\nu_2)\left(\m{\nu_1}+\m{\nu_2}\right).
\end{align*}
\end{remark}

We have also the following relation.
\begin{lemma}
If $\phi:\prob_2(\R^d)\to\R$ is $\mm$-locally Lipschitz, then it is locally Lipschitz continuous.
\end{lemma}
\begin{proof}
Let $\nu_1\in\prob_2(\R^d)$ and fix $R_1>0$. Define $R:=R_1+\m{\nu_1}$ and $L_R$ be as in Definition \ref{def:locLip}\eqref{item:m2lip}. Then, taken $\nu_2\in\prob_2(\R^d)$ such that $W_2(\nu_1,\nu_2)\le R_1$, we have
\[\m{\nu_2}=W_2(\nu_2,\delta_0)\le W_2(\nu_1,\nu_2)+\m{\nu_1}\le R_1+\m{\nu_1}=R.\]
We can thus apply the $\mm$-local Lipschitzianity of $\phi$ to conclude
\[|\phi(\nu_1)-\phi(\nu_2)|\le L_R\, W_2(\nu_1,\nu_2).\]
\end{proof}

\begin{theorem}[Comparison]\label{thm:comparison}
Let $u_i:[0,T]\times\prob_2(\R^d)\to\R$, $i=1,2$, be such that
\begin{enumerate}
\item\label{item:comp1} $t\mapsto u_i(t,\nu)$ is continuous in $[0,T]$ for all $\nu\in\prob_2(\R^d)$, $i=1,2$;
\item\label{item:comp2} $\nu\mapsto u_i(t,\nu)$ is $\mm$-locally Lipschitz in $\prob_2(\R^d)$, uniformly w.r.t. $t$, i.e. 
for any $R>0$ there exists $L_{R,i}>0$ such that
\begin{equation*}
|u_i(t,\nu_1)-u_i(t,\nu_2)|\le L_{R,i}\, W_2(\nu_1,\nu_2),
\end{equation*}
for any $\nu_i\in\prob_2(\R^d)$ such that $\m{\nu_i}\le R$, $i=1,2$, for any $t\in[0,T]$.
\item\label{item:comp3} $u_1$ and $u_2$ are respectively a viscosity subsolution and supersolution of 
\begin{equation}\label{eq:HJBcomp}
\partial_t U(t,\nu)+\H(\nu,D_\nu U(t,\nu))=0
\end{equation}
according with Definition \ref{def:viscosol}, where $\H(\nu,p)$ is defined for any $\nu\in\prob_2(\R^d)$, $p\in L^2_\nu(\R^d)$.
\end{enumerate}
Assume that there exists a continuous nondecreasing map $\omega_{\H}:\R^2\to[0,+\infty)$ such that $\omega_{\H}(0,0)=0$ and
\begin{equation}\label{eq:assHCP}
|\H(\nu_1,\lambda\, p_{\nu_1,\nu_2})-\H(\nu_2,\lambda\, q_{\nu_1,\nu_2})|\le \omega_{\H}\left(W_2(\nu_1,\nu_2),\,\lambda\,W_2^2(\nu_1,\nu_2)\right),
\end{equation}
for all $\lambda>0$, $\nu_1,\nu_2\in\prob_2(\R^d)$, with $p_{\nu_1,\nu_2}$ and $q_{\nu_1,\nu_2}$ as in \eqref{def:pq}. Then
\[\inf_{(t,\nu)\in[0,T]\times\prob_2(\R^d)}\left\{u_2(t,\nu)-u_1(t,\nu)\right\}=\inf_{\nu\in\prob_2(\R^d)}\left\{u_2(T,\nu)-u_1(T,\nu)\right\}.\]
\end{theorem}

\section{Reachability through Lyapunov functions}\label{sec:reach}

\subsection{Lyapunov function and asymptotic reachability}

{\color{black}We introduce Lyapunov functions satisfying the $\mm$-local Lipschitzianity property introduced in Definition \ref{def:locLip}.}
 
In the following, $\F$ is the set-valued map of Section \ref{sec:adm} and $\HF$ is given by \eqref{eq:H}.
\begin{definition}[Lyapunov function]
Given $\varphi:[0,+\infty)\to\R$, we say that a $\mm$-locally Lipschitz map $\V:\prob_2(\R^d)\to[0,+\infty)$ is a \emph{Lyapunov function for $\F$ with rate $\varphi$} if there exists a map $C:\prob_2(\R^d)\to(0,+\infty)$ bounded on bounded sets such that the following \emph{Hamilton-Jacobi inequality} holds
\begin{equation}\label{eq:Linequality}
\varphi(\V(\nu))+\HF(\nu,p)\le C(\nu)\delta,\qquad\forall\,\nu\in\prob_2(\R^d),\,\forall\,p\in\partial_\delta^-\V(\nu),\,\forall\,\delta>0,
\end{equation}
where $\HF$ is the Hamiltonian given in \eqref{eq:H}.
\end{definition}
{\color{black}In the whole paper, we consider the} following function $\varphi:[0,+\infty)\to\R$, that will be associated with the characterization of asymptotic reachability (see next Definition \ref{def:reach+asym}):
\begin{equation}\label{eq:phi}
\varphi(y):=\alpha y,
\end{equation}
where $\alpha>0$ is a given parameter.

\begin{definition}\label{def:reach+asym}
Let $\nu\in\prob_2(\R^d)$, we say that
\begin{enumerate}
\item $\nu$ is \emph{weakly asymptotically reachable} if for any $\mu_0\in\prob_2(\R^d)$ there exists $\mu\in\adm{[0,+\infty)}{\mu_0}$ such that $\mu_t\weakto\nu$ narrowly as $t\to+\infty$;
\item $\nu$ is \emph{strongly asymptotically reachable} if for any $\mu_0\in\prob_2(\R^d)$ there exists $\mu\in\adm{[0,+\infty)}{\mu_0}$ such that $\mu_t\overset{W_2}{\to}\nu$ as $t\to+\infty$, i.e. $W_2(\mu_t,\nu)\to0$.
\end{enumerate}
\end{definition}

Inspired by the classical finite-dimensional setting presented in \cite{Au}, we give sufficient conditions for asymptotic reachability via Lyapunov functions using a viscosity-type approach. A viscosity approach for control problems in Wasserstein spaces is used also in \cite{CMQ} concerning the problem of viability/invariancy.

\begin{lemma}\label{lem:Lvisco}
Let $\V:\prob_2(\R^d)\to[0,+\infty)$ be a $\mm$-locally Lipschitz map and $T>0$. The map $\V$ is a Lyapunov function for $\F$ with rate $\varphi$ {\color{black}as in \eqref{eq:phi}} if and only if the function $\S: [0,T]\times\prob_2(\R^d)\to[0,+\infty)$, defined by $\S(t,\nu):=e^{\alpha\, t}\V(\nu)$, is a \emph{partial-viscosity supersolution} of
\begin{equation}\label{eq:HJB}
\partial_t U(t,\nu)+\HF(\nu,D_\nu U(t,\nu))=0,
\end{equation}
i.e. there exists a map $C_T:\prob_2(\R^d)\to(0,+\infty)$, $C_T(\cdot)$ bounded on bounded sets, such that
\begin{equation}\label{eq:S2-1}
\partial_t\S(t,\nu)+\HF(\nu,\xi)\le C_T(\nu)\,\delta,\quad\forall\, (t,\,\nu)\in(0,T)\times\prob_2(\R^d),\,\forall\,\xi\in\partial_\delta^-\S(t,\nu),\,\forall\,\delta>0.
\end{equation}
\end{lemma}
\begin{proof}

Notice that $\partial_t\S(t,\nu)=\alpha e^{\alpha\, t}\,\V(\nu)$ and, by Definition \ref{def:diffP2}, we have
\begin{equation}\label{eq:partialSetV}
\partial_\delta^-\S(t,\nu)=e^{\alpha\, t}\,\partial_{\delta e^{-\alpha t}}^-\V(\nu),\quad\forall(t,\nu)\in[0,T]\times\prob_2(\R^d).
\end{equation}
Then the relation \eqref{eq:S2-1} is equivalent to the following condition
\begin{equation*}
\alpha e^{\alpha\, t}\,\V(\nu)+e^{\alpha\, t}\HF(\nu,p)\le C_T(\nu)\,\delta,\quad\forall\, (t,\,\nu)\in(0,T)\times\prob_2(\R^d),\,\forall\,p\in\partial_{\delta e^{-\alpha t}}^-\V(\nu),\,\forall\,\delta>0.
\end{equation*}
Dividing by $e^{\alpha\, t}$ and since the condition holds for any $\delta>0$, we get the equivalence with
\begin{equation*}
\alpha\V(\nu)+\HF(\nu,p)\le C_T(\nu)\,\tilde\delta,\quad\forall\, \nu\in\prob_2(\R^d),\,\forall\,p\in\partial_{\tilde\delta}^-\V(\nu),\,\forall\,\tilde\delta>0,
\end{equation*}
i.e. $\V$ is a Lyapunov function for $\F$ with rate $\varphi$.
\end{proof}

We point out that, even if $U$ is partially differentiable w.r.t. $t\in(a,b)$, there is no apparent relation in general between being a viscosity supersolution according with Definition \ref{def:viscosol} and being a partial viscosity supersolution as in Lemma \ref{lem:Lvisco}.
Of course, if we modify the conditions in Definitions \ref{def:viscosol} and Lemma \ref{lem:Lvisco} to hold only for $\delta=0$, then by Remark \ref{rmk:problem1} we have that being a partial supersolution is stronger than being a supersolution of \eqref{eq:HJBintro}.

However, even without this requirement on $\delta$, we can still prove the following result.
\begin{lemma}\label{Q1}
Let $T>0$ and $\V:\prob_2(\R^d)\to[0,+\infty)$ be a Lyapunov function for $\F$ with rate $ \varphi$ {\color{black} as in \eqref{eq:phi}}, namely $\V$ is $\mm$-locally Lipschitz and there exists a map $C:\prob_2(\R^d)\to(0,+\infty)$ bounded on bounded sets such that the following holds
\begin{equation}\label{eq:Linequality'}
\alpha\V(\nu)+\HF(\nu,p)\le C(\nu)\delta,\qquad\forall\,\nu\in\prob_2(\R^d),\,\forall\,p\in\partial_\delta^-\V(\nu),\,\forall\,\delta>0.
\end{equation}

 Then the function $\S: [0,T]\times\prob_2(\R^d)\to[0,+\infty)$, defined by $\S(t,\nu):=e^{\alpha\, t}\V(\nu)$, is a viscosity supersolution of
\begin{equation}\label{eq:HJB'}
\partial_t U(t,\nu)+\HF(\nu,D_\nu U(t,\nu))=0,
\end{equation}
in the sense of Definition \ref{def:viscosol}.
\end{lemma}
\begin{proof} Let us fix $ \delta >0$, $(t, \nu)\in[0,T]\times\prob_2(\R^d)$ and  $( p_t, p_ \nu) \in \boldsymbol{D}^-_\delta \S(t,\nu)$; we will prove that 
\begin{equation}\label{a}
p_t + \HF ( \nu , p _ \nu )  \leq (  C(\nu) e ^{\alpha T } +1) \delta,
\end{equation}
which implies that $\S$  is a viscosity supersolution  to (\ref{eq:HJB'}). 
By Proposition \ref{prop:incldiff2.0} and recalling the equivalence \eqref{eq:partialSetV}, there exist $a_t \in [-1,1] $ and  $ \bar p _ \nu \in   \partial_{ e ^{- \alpha t } \delta} ^-\V(\nu) \subset  \partial_{   \delta }^-\V(\nu) $  such that
\begin{equation*}\left\{ \begin{array}{l}
p_t = \alpha e ^{\alpha t} V (\nu) + \delta a_t \\
p _\nu = e ^{\alpha t} \bar p _ \nu.
\end{array}\right.
\end {equation*}
So that
 \begin{align*}
p_t + \HF ( \nu , p _ \nu ) &=
 \delta a_t  + \alpha e ^{\alpha t} V (\nu)  +  \HF ( \nu , e ^{\alpha t} \bar p _ \nu ) \\
&\leq \delta +  e ^{\alpha t} C(\nu) \delta \leq  (  C(\nu) e ^{\alpha T } +1) \delta,
\end{align*}
since \eqref{eq:Linequality'} holds for $\bar p _ \nu $. Hence we get (\ref{a}) and the proof is complete.
\end{proof}
\begin{proposition}\label{prop:decreas}
Assume Hypothesis \ref{hypo}. Let $\V:\prob_2(\R^d)\to[0,+\infty)$ be  $\mm$-locally Lipschitz. The map $\V$ is a Lyapunov function for $\F$ with rate $\varphi$ {\color{black}as in \eqref{eq:phi}} if and only if for any $\mu_0\in\prob_2(\R^d)$ there exists $\mu\in\adm{[0,+\infty)}{\mu_0}$ such that the function $\SS:[0,+\infty)\to[0,+\infty)$, defined by $\SS(t):=e^{\alpha\,t}\V(\mu_t)$, is nonincreasing.
\end{proposition}

\begin{proof}
First, assume that $\V$ is a Lyapunov function for $\F$ with rate $\varphi$. Fix $\mu_0\in\prob_2(\R^d)$, $T>0$ and consider the functional $g:\prob_2(\R^d)\to\R$ defined by
\[g(\nu):=e^{\alpha\, T}\V(\nu),\quad \text{for any }\nu\in\prob_2(\R^d).\]
Let $U_g:[0,T]\times\prob_2(\R^d)\to\R$, defined by
\begin{equation}\label{eq:Wcost}
U_g(t,\nu):=\inf\left\{g(\mu_{T})\,:\,\mu\in\adm{[t,T]}{\nu}\right\}.
\end{equation}
Since $g$ is $\mm$-locally Lipschitz, by Lemma \ref{lem:LipUg} the function $U_g$ is continuous in $[0,T]$, $\mm$-locally Lipschitz in $\prob_2(\R^d)$ and it is a viscosity solution of
\begin{equation}\label{syst:HJB+FC}
\begin{cases}
\partial_t U(t,\nu)+\HF(\nu,D_\nu U(t,\nu))=0,\\
U(T,\nu)=g(\nu),
\end{cases}
\end{equation}
according with Definition \ref{def:viscosol}.
Moreover, by Lemma \ref{Q1}, the function $\S$ is a viscosity supersolution of \eqref{syst:HJB+FC}.
Recalling Lemma \ref{lem:Hok},  by the  Comparison Principle  (Theorem \ref{thm:comparison}) we get 
\[U_g(t,\nu)\le \S(t,\nu),\quad\forall\,(t,\nu)\in[0,T]\times\prob_2(\R^d).\]
As recalled in Section \ref{sec:adm}, by the Dynamic Programming Principle (cf. \cite[Corollary 2 and Proposition 3]{JMQ20}), there exists $\mu\in\adm{[0,T]}{\mu_0}$ such that $U_g(t,\mu_t)= U_g(s,\mu_s)$ for all $0\le s\le t\le T$. Thus, we have
\begin{equation}\label{eq:DPP+compfor2}
e^{\alpha\,T}\V(\mu_{T})=U_g(T,\mu_{T})= U_g(s,\mu_s)\le \S(s,\mu_s)= e^{\alpha\,s}\V(\mu_s),\quad \forall\, s\in[0,T].
\end{equation}
By the arbitrariness of $T>0$ and since the gluing of admissible trajectories is still an admissible trajectory, we conclude that there exists $\mu\in\adm{[0,+\infty)}{\mu_0}$ such that the map $\SS$, is nonincreasing in $[0,+\infty)$. This proves the left-to-right implication in Proposition \ref{prop:decreas}.

We now prove the reverse implication. Let $T>0$ be arbitrarily fixed.
In view of the equivalence result of Lemma \ref{lem:Lvisco}, in order to prove that $\V$ is a Lyapunov function for $\F$ with rate $\varphi$, we proceed as in the proof of \cite[Theorem 3, Claim 2]{JMQ20}. 
Fix $(\bar{t},\bar{\nu})\in(0,T)\times\prob_2(\R^d)$, $\delta>0$, and take  $p_{\bar{\nu}}\in\partial_\delta^-\V( \bar \nu) $. By assumption, up to a time variable change, there exists $\mu\in\adm{[\bar{t},+\infty)}{\bar{\nu}}$ such that $\S(t,\mu_t)\le\S(s,\mu_s)$ for all $\bar{t}\le s \le t\le T$.

By Definition \ref{def:diffP2}, for all $\ppi_t\in\Gamma(\bar{\nu},\mu_t)$ we have
\begin{align*}
0&\ge \S(t,\mu_t)-\S(\bar{t},\bar{\nu}) = e^{\alpha t } \V (\mu _t) -  e^{\alpha \bar t } \V (\bar{\nu} ) \\ 
& = (e^{\alpha t }  -  e^{\alpha \bar t })  \V (\mu _t) +  e^{\alpha \bar t }( \V (\mu _t) -
 \V (\bar{\nu} ) ) \\
&\ge  (e^{\alpha t }  -  e^{\alpha \bar t })  \V (\mu _t) +e^{\alpha\bar{t}}\left(\int\langle p_{\bar{\nu}}(x),y-x\rangle \d\ppi_t(x,y)-\delta\cdot\Delta_{\ppi_t}+o(\Delta_{\ppi_t})\right)
\end{align*}
where $\Delta_{\ppi_t}:=\sqrt{\int|x-y|^2\d\ppi_t}$.
We divide by $ |t - \bar t| $ and pass to the $\liminf_{t\to\bar{t}^+}$; using the continuity of $s\mapsto\mu_s$ and of $V$, we get 
\begin{equation*}
0\ge \alpha e^{\alpha\bar{t}}V(\bar{\nu})+e^{\alpha\bar{t}}\liminf_{t\to\bar{t}^+}\frac{1}{t-\bar{t}}\left(\int\langle p_{\bar{\nu}}(x),y-x\rangle \d\ppi_t(x,y)-\delta\cdot\Delta_{\ppi_t}+o(\Delta_{\ppi_t})\right),
\end{equation*}
hence
\begin{equation}\label{eq:quasiend}
{\color{black}\limsup_{t\to\bar{t}^+}}\frac{\Delta_{\ppi_t}}{t-\bar{t}}\left(\delta-\frac{o(\Delta_{\ppi_t})}{\Delta_{\ppi_t}}\right)\ge\alpha V(\bar{\nu})+\liminf_{t\to\bar{t}^+}\frac{1}{t-\bar{t}}\int\langle p_{\bar{\nu}}(x),y-x\rangle \d\ppi_t(x,y).
\end{equation}

Let $\ev_t:\R^d\times\rmC([\bar{t},T])\to\R^d$ be the evaluation operator at time $t\in[\bar{t},T]$, defined by $\ev_t(x,\gamma):=\gamma(t)$. Let $\eeta\in\prob(\R^d\times\rmC([\bar{t},T]))$ be such that $\mu_t={\ev_t}_\sharp\eeta$ and such that for $\eeta$-a.e. $(x,\gamma)\in\R^d\times\rmC([\bar{t},T])$ we have $\gamma(\bar{t})=x$ and $\gamma$ is an absolutely continuous solution of $\dot\gamma(t)=v_t(\gamma(t))$, where $v_t$ is an admissible Borel vector field driving $\mu$, i.e. $v_t\in L^2_{\mu_t}(\R^d;\R^d)$,
 $v_t(x)\in\F(x,\mu_t)$ and $\partial_t\mu_t+\div(v_t\,\mu_t)=0$ for $t>\bar{t}$. We recall that such a measure $\eeta$ exists in view of the Superposition Principle \cite[Theorem 8.2.1]{ags} and thanks to the growth property of $\F$ in \eqref{eq:growthF}. Define the plan $\tilde{\ppi}_t:=(\ev_{\bar{t}},\ev_t)_\sharp\eeta\in\Gamma(\bar{\nu},\mu_t)$, then by Jensen's inequality we have
\begin{align*}
\frac{1}{(t-\bar{t})^2}\int_{\R^d\times\R^d}|x-y|^2\d\tilde{\ppi}_t(x,y)&=\int_{\R^d\times\rmC([\bar{t},T])}\left|\frac{\gamma(t)-\gamma(\bar{t})}{t-\bar{t}}\right|^2\d\eeta(x,\gamma)\\
&=\int_{\R^d\times\rmC([\bar{t},T])}\left|\frac{1}{t-\bar{t}}\int_{\bar{t}}^t\dot\gamma(s)\d s\right|^2\d\eeta(x,\gamma)\\
&\le\int_{\R^d\times\rmC([\bar{t},T])}\frac{1}{t-\bar{t}}\int_{\bar{t}}^t|\dot\gamma(s)|^2\d s\d\eeta(x,\gamma)\\
&\le\int_{\R^d\times\rmC([\bar{t},T])}\|\dot\gamma\|^2_{L^\infty([\bar{t},T])}\d\eeta\le D(\bar{\nu}),
\end{align*}
where $D(\bar{\nu})$ comes from Lemma \ref{lem:attainCM}. We then get
\begin{equation}\label{eq:estimDeltat}
\frac{\Delta_{\tilde{\ppi}_t}}{t-\bar{t}}=\sqrt{\frac{1}{(t-\bar{t})^2}\int |x-y|^2\d\tilde{\ppi}_t}\le\sqrt{D(\bar{\nu})}=:C(\bar{\nu}),
\end{equation}
where $C:\prob_2(\R^d)\to(0,+\infty)$ is bounded on bounded sets.

We now analyze the right-hand side in \eqref{eq:quasiend} with the choice of $\tilde{\ppi}_t$. We have
\begin{align*}
\frac{1}{t-\bar{t}}\int_{\R^d\times\R^d}\langle p_{\bar{\nu}}(x),y-x\rangle \d\tilde{\ppi}_t(x,y)&=\int_{\R^d\times\rmC([\bar{t},T])}\langle p_{\bar{\nu}}(x),\frac{\gamma(t)-\gamma(\bar{t})}{t-\bar{t}}\rangle\d\eeta(x,\gamma)\\
&=\int_{\R^d\times\rmC([\bar{t},T])}\langle p_{\bar{\nu}}(x),\frac{1}{t-\bar{t}}\int_{\bar{t}}^t\dot\gamma(s)\d s\rangle\d\eeta(x,\gamma)\\
&\ge\int_{\R^d\times\rmC([\bar{t},T])}\frac{1}{t-\bar{t}}\int_{\bar{t}}^t\inf_{v\in\F(\gamma(s),\mu_s)}\langle p_{\bar{\nu}}(x),v\rangle\d s\d\eeta(x,\gamma)\\
&\ge \int_{\R^d\times\rmC([\bar{t},T])}\frac{1}{t-\bar{t}}\int_{\bar{t}}^t\left(\inf_{v\in\F(x,\bar{\nu})}\langle p_{\bar{\nu}}(x),v\rangle-\delta_s|p_{\bar{\nu}}(x)|\right)\d s\d\eeta(x,\gamma),
\end{align*}
where $\delta_s=L\,(W_2(\mu_s,\bar{\nu})+|\gamma(s)-x|)$. By taking the $\liminf_{t\to\bar{t}^+}$ and applying Fatou's Lemma (recall the estimate in Lemma \ref{lem:attainCM}, that $p_{\bar{\nu}}\in L^2_{\bar{\nu}}$ and the growth property of $\F$ in \eqref{eq:growthF}), we have
\begin{align*}
&\liminf_{t\to\bar{t}^+}\frac{1}{t-\bar{t}}\int_{\R^d\times\R^d}\langle p_{\bar{\nu}}(x),y-x\rangle \d\tilde{\ppi}_t(x,y)\ge\\
&\ge \int_{\R^d\times\rmC([\bar{t},T])}\liminf_{t\to\bar{t}^+}\frac{1}{t-\bar{t}}\int_{\bar{t}}^t\left(\inf_{v\in\F(x,\bar{\nu})}\langle p_{\bar{\nu}}(x),v\rangle-\delta_s|p_{\bar{\nu}}(x)|\right)\d s\d\eeta(x,\gamma)\\
&=\int_{\R^d}\inf_{v\in\F(x,\bar{\nu})}\langle p_{\bar{\nu}}(x),v\rangle\d\bar{\nu}(x),
\end{align*}
where the last equality follows by the continuity of $s\mapsto\delta_s$. Thus, by \eqref{eq:quasiend} applied to the plan $\tilde{\ppi}_t$, using \eqref{eq:estimDeltat} and the estimate above, we get
\begin{equation}\label{eq:almostfinal}
C(\bar{\nu})\delta\ge\alpha \V(\bar{\nu})+\int_{\R^d}\inf_{v\in\F(x,\bar{\nu})}\langle p_{\bar{\nu}}(x),v\rangle\d\bar{\nu}(x)=\varphi(\V(\bar{\nu}))+\HF(\bar{\nu},p_{\bar{\nu}}).
\end{equation}
We conclude that $\V$ is a Lyapunov function for $\F$ with rate $\varphi$.
\end{proof}

Following \cite{Au,AuC,AuF}, we can give another characterization for Lyapunov functions in terms of viability conditions over their epigraph, as follows.

We denote by $\epi{U}$ the \emph{epigraph} of a function $U:\prob_2(\R^d)\to\R$, i.e. the set
\[\epi{U}:=\left\{(\nu,r)\in\prob_2(\R^d)\times\R\,:\,r\ge U(\nu)\right\}.\]
\begin{corollary}\label{cor:Lviab}
Assume Hypothesis \ref{hypo}. Let $\V:\prob_2(\R^d)\to[0,+\infty)$ be a $\mm$-locally Lipschitz map. Then $\V$ is a Lyapunov function for $\F$ with rate $\varphi$ {\color{black}as in \eqref{eq:phi}} if and only if $\epi{\V}$ is viable for the coupled dynamics
\begin{equation}\label{eq:CE+DE}
\begin{cases}
\partial_t\mu_t+\div(v_t\,\mu_t)=0, &t>0,\quad\text{with }
v_t(x)\in\F(x,\mu_t),\\
y'(t)=-\varphi(y(t)), &t>0.
\end{cases}
\end{equation}
\end{corollary}
\begin{proof}
Recall that, by definition, $\epi{\V}$ is viable for the coupled dynamics \eqref{eq:CE+DE} if and only if the following condition holds
\begin{equation*} \begin{array}{ll} 
(\mathrm{V1}) &
\begin{array}{l} 
\mbox{$\forall\,(\mu_0,y_0)\in\epi{\V}$ there exists a solution $(\mu,y)$ of \eqref{eq:CE+DE} with initial condition $(\mu_0,y_0)$,} \\ \mbox{i.e. $\mu\in\adm{[0,+\infty)}{\mu_0}$ and $y(t;y_0)=e^{-\alpha\,t}y_0$, such that $\left(\mu_t,y(t;y_0)\right)\in\epi{\V}$ for all $t\ge0$,}\\ \mbox{i.e. $\V(\mu_t)\le e^{-\alpha\,t}y_0$ for all $t\ge 0$.} \end{array}  \end{array}
\end{equation*}
Notice that condition $(\mathrm{V1})$  is equivalent to the following one, where the initial condition is taken only at the boundary of $\epi{\V}$:
\begin{equation*} \begin{array}{ll} 
(\mathrm{V2}) &
\begin{array}{l} 
\mbox{$\forall\,\mu_0\in\prob_2(\R^d)$ there exists $\mu\in\adm{[0,+\infty)}{\mu_0}$ such that} \\ 
\mbox{$\left(\mu_t,y(t;\V(\mu_0)\right)=\left(\mu_t,e^{-\alpha\,t}\V(\mu_0)\right)\in\epi{\V}$ for all $t\ge0$,}\\ \mbox{i.e. $e^{\alpha\,t}\V(\mu_t)\le\V(\mu_0)$ for all $t\ge 0$. }
 \end{array}  \end{array}
\end{equation*}
Thanks to Proposition \ref{prop:decreas}, to conclude we only need to prove  that $(\mathrm{V2})$  is equivalent to 
\begin{equation}\label{V3} \begin{array}{l} 
\mbox{$\forall\,\mu_0\in\prob_2(\R^d)$ there exists $\mu\in\adm{[0,+\infty)}{\mu_0}$ such that} \\ 
\mbox{$\SS(t):=e^{\alpha\,t}\V(\mu_t)$, is nonincreasing. }\end{array}
\end{equation} 
Clearly (\ref{V3}) implies $(\mathrm{V2})$.
 Conversely, let us assume that $(\mathrm{V2})$ holds true. Fix $\mu_0\in\prob_2(\R^d)$ and $T >0$. Consider for any integer $n>0$ the subdivision of $[0,T]$ given by $ k \frac{T}{n} $, for $k=0,1, \ldots, n$.
By $(\mathrm{V2})$ there exists $\mu ^n\in\adm{[0, \frac{T}{n} ] }{\mu_0}$ such that $$ e ^{\alpha t } V (\mu ^n _t) \leq V (\mu_0), \; \forall t \in \left[0,  \frac{T}{n}\right].$$ 
Applying again $(\mathrm{V2})$  with the initial condition $ \mu^n_{\frac{T}{n}}$, we can extend $ \mu ^n$ on $\left[\frac{T}{n}, 2 \frac{T}{n}\right]$ such that 
$$ e ^{\alpha( t -\frac{T}{n})  } V (\mu ^n _t) \leq V (\mu_{\frac{T}{n}} ^n ), \; \forall t \in \left[\frac{T}{n}, 2 \frac{T}{n}\right].$$ By a direct iteration, we have obtained $\mu ^n\in\adm{[0, T ] }{\mu_0}$  such that 
\begin{equation} \label{V4} e ^{\alpha( t -k\frac{T}{n})  } V (\mu ^n _t) \leq V (\mu_{k\frac{T}{n}} ^n ), \; \forall t \in \left[k\frac{T}{n}, (k+1) \frac{T}{n}\right], \; \forall k = 0,1, \ldots (n-1).\end{equation}
By relative compactness of $\adm{[0, T ] }{\mu_0}$ (cf. \cite[Corollary 1]{JMQ20}), we have that $\{\mu ^n\}_n\subset\adm{[0, T ] }{\mu_0}$ converges to some $ \mu\in \adm{[0, T ] }{\mu_0}$ uniformly on $ [0,T]$, up to subsequences.

Fix $0 < \sigma < \tau <T$, we claim that $ e ^{\alpha \sigma} V (\mu _ \sigma) \ge e ^{ \alpha \tau } V ( \mu _ \tau) $. Indeed, for $n $ large enough, we can find integers $ k^n _\sigma < k ^n _\tau \leq T$ such that
$$ 0 < k ^n  _ \sigma \frac{T}{n} \leq \sigma <(1+ k ^n  _ \sigma )\frac{T}{n} < k ^n  _ \tau \frac{T}{n} \leq  \tau <(1+ k ^n  _ \tau )  \frac{T}{n} < T.$$
Denoting $ \sigma _n : =  k ^n  _ \sigma \frac{T}{n}$ and 
$\tau _n :=  k ^n  _ \tau \frac{T}{n}$, a straightforward use of (\ref{V4}) gives 
\[ e ^{ \alpha (\tau _n  - \sigma _n) } V ( \mu ^n _{\tau _n}) \leq V (\mu ^n _{\sigma _n}).\]
Passing to the limit as $ n \to +\infty $, the above inequality gives our claim  $ e ^{\alpha \sigma} V (\mu _ \sigma) \ge e ^{ \alpha \tau } V ( \mu _ \tau) $.
Thus $t \mapsto e^{\alpha\,t}\V(\mu_t)$, is nonincreasing on $[0,T]$. Applying the same argument on the intervals $[iT,(i+1)T]$, for all $i\in\N$, we obtain \eqref{V3}.
\end{proof}

In the following, we use Proposition \ref{prop:decreas} to obtain the main result of the paper: a sufficient condition for asymptotic reachability via Lyapunov functions.
\begin{theorem}[Asymptotic reachability]\label{thm:main}
Assume Hypothesis \ref{hypo}. Let $\bar\nu\in\prob_2(\R^d)$ and assume there exists a Lyapunov function $\V$ for $\F$ with rate $\varphi$ {\color{black}as in \eqref{eq:phi}} such that $\V(\bar\nu)=0$, $\V(\nu)>0$ for any $\nu\neq\bar\nu$ and satisfying either one of the following conditions: 
\begin{enumerate}
\item\label{item:1} $\V$ is narrowly lower semicontinuous and if $\{\nu_n\}_{n\in\N}\subset\prob_2(\R^d)$ s.t. $\m{\nu_n}\to+\infty$, then $\V(\nu_n)\to+\infty$, as $n\to+\infty$;
\item\label{item:2} $\V$ has compact sublevels in $W_2$-topology;
\item\label{item:3} there exists $\eps>0$ such that, for any $\{\nu_n\}_{n\in\N}\subset\prob_2(\R^d)$ with $\meps{\nu_n}\to+\infty$, we have $\V(\nu_n)\to+\infty$, as $n\to+\infty$.
\end{enumerate}
If \eqref{item:1} holds, then $\bar\nu$ is weakly asymptotically reachable; while
if either \eqref{item:2} or \eqref{item:3} holds, then $\bar\nu$ is strongly asymptotically reachable.
\end{theorem}
\begin{proof}
Let $\mu_0\in\prob_2(\R^d)$. By Proposition \ref{prop:decreas}, there exists $\mu\in\adm{[0,+\infty)}{\mu_0}$ such that $0\le\V(\mu_t)\le e^{-\alpha\,t}\V(\mu_0)$, for all $t\ge0$. In particular,
\begin{equation}\label{eq:dec2}
\V(\mu_t)\to 0,\text{ as }t\to+\infty.
\end{equation}
i.e. given $K>0$ there exists $\bar{t}>0$ such that $\V(\mu_t)\le K$ for any $t\ge\bar{t}$.
\begin{enumerate}
\item Assume \eqref{item:1} holds. Then, by \eqref{eq:dec2} there exists $\tilde{K}>0$ such that $\m{\mu_t}\le \tilde{K}$ for all $t\ge0$, hence $\{\mu_t\}_{t\ge0}$ is tight thanks to \cite[Remark 5.1.5]{ags} with $\phi(x)=|x|^2$. By Prokhorov's compactness theorem \cite[Theorem 5.1.3]{ags}, there exists $\{t_n\}_{n\in\N}$ and $\bar{\mu}\in\prob(\R^d)$ such that $t_n\to+\infty$ as $n\to+\infty$ and $\mu_{t_n}\weakto\bar{\mu}$ narrowly as $n\to+\infty$. Moreover, since $\prob(\R^d)\ni\nu\mapsto \m{\nu}$ is l.s.c. with respect to narrow convergence (see \cite[Lemma 5.1.7]{ags}), then in particular $\bar{\mu}\in\prob_2(\R^d)$. Since $\V$ is narrowly lower semicontinuous, by \eqref{eq:dec2} we get that
\[0=\lim_{n\to+\infty}\V(\mu_{t_n})\ge\V(\bar{\mu})\ge0,\]
hence $\bar{\mu}=\bar{\nu}$ and we conclude that $\bar{\nu}$ is weakly asymptotically reachable.
\item Assume \eqref{item:2} holds. By \eqref{eq:dec2}, $\{\mu_t\}_{t\ge\bar{t}}$ is a $W_2$-closed subset of the $W_2$-compact set $\{\nu\in\prob_2(\R^d)\,:\,\V(\nu)\le K\}$, hence it is compact as well. Thus, by continuity of $V$ in $W_2$ and reasoning as in the conclusion of the previous item, there exists $\{t_n\}_{n\in\N}$ such that $t_n\to+\infty$ as $n\to+\infty$ and $\mu_{t_n}\overset{W_2}{\to}\bar{\nu}$ as $n\to+\infty$. This proves that $\bar\nu$ is strongly asymptotically reachable.
\item Assume \eqref{item:3} holds. Then, \eqref{eq:dec2} implies that there exists $\tilde{K}>0$ such that $\meps{\mu_t}\le \tilde{K}$ for all $t\ge0$. Reasoning as in item (1), there exists $\{t_n\}_{n\in\N}$ and $\bar{\mu}\in\prob(\R^d)$ such that $t_n\to+\infty$ as $n\to+\infty$ and $\mu_{t_n}\weakto\bar{\mu}$ narrowly as $n\to+\infty$. Moreover, $\{\mu_t\}_{t\ge0}$ has uniformly integrable $2$-moments (see \cite[Equation (5.1.20)]{ags}), hence $\mu_{t_n}\overset{W_2}{\to}\bar{\mu}$ as $n\to+\infty$. We conclude by the continuity of $V$ in $W_2$-metric.
\end{enumerate}
\end{proof}


{\color{black}
\section{Applications}\label{sec:examples}
In this section we give two examples of application of our main result.

\subsection{The second moment as a Lyapunov function}
Let $\alpha>0$ and consider the following set-valued vector field $\F:\R^d\times\prob_2(\R^d)\tto\R^d$ driving the dynamics in Definition \ref{def-macro},
\[\F(x,\nu):=\overline{B\big(0,\alpha\left(|x|+\m{\nu}\right)\big)},\]
so that $\F$ satisfies Hypothesis \ref{hypo}.
We can prove that $\bar{\nu}:=\delta_0$ is weakly asymptotically reachable thanks to Theorem \ref{thm:main}. Indeed, we can show that the function $\V:\prob_2(\R^d)\to[0,+\infty)$, defined by $\V(\nu):=\frac{1}{2}\mm^2(\nu)$, is a Lyapunov function for $\F$ with rate $\varphi$ as in \eqref{eq:phi}, satisfying the required conditions. In particular, by Proposition \ref{prop:decreas}, in order to prove that $\V$ is a Lyapunov function for $\F$ with rate $\varphi$ it is sufficient to prove the following claims:
\begin{itemize}
\item \emph{Claim 1:} $\V$ is $\mm$-locally Lispchitz.
\smallskip

\item \emph{Claim 2:} for any $\mu_0\in\prob_2(\R^d)$ there exists $\mu\in\adm{[0,+\infty)}{\mu_0}$ such that the function $\SS:[0,+\infty)\to[0,+\infty)$, defined by $\SS(t):=e^{\alpha\,t}\V(\mu_t)$, is nonincreasing.
\end{itemize}

Claim 1 is proved in Remark \ref{rmk:m2}. In order to prove Claim 2, we proceed by proving that for any $\tilde\mu_0\in D$ there exists $\tilde\mu\in\adm{[0,+\infty)}{\mu_0}$ such that the function $\SS$ is nonincreasing, where
\[D:=\left\{\nu=\sum_{i=1}^{N}\beta_i\,\delta_{x^i}\,:\,N\in\N,\,(x^i)_i\subset\R^d,\,\beta_i\in[0,1],\,\sum_i\beta_i=1\right\}.\]
Then, Claim 2 follows by density of $D$ into $\prob_2(\R^d)$ w.r.t. the metric $W_2$, relative compactness of $\displaystyle\bigcup_{\nu\in\prob_2(\R^d),\m{\nu}\le C}\adm{[0,T]}{\nu}$ for any $C,T>0$ and continuity of $\V$. 
Let $N\in\N$, $(\beta_i)_{i=1,\dots,N}\subset[0,1]$ with $\sum_{i=1}^N \beta_i=1$ and $(x_0^i)_{i=1,\dots,N}\subset\R^d$. Take $\tilde\mu_0=\sum_{i=1}^N \beta_i\delta_{x_0^i}$, then the trajectory $\tilde\mu=(\tilde\mu_t)_t$, with
\[\tilde\mu_t:=\sum_{i=1}^N \beta_i\delta_{e^{-\alpha t}x_0^i}\]
is admissible for $\F$ starting from $\tilde\mu_0$. Moreover, for any $t\ge0$ we have
\begin{equation*}
\V(\tilde\mu_t)=\frac{1}{2}\mm^2(\tilde\mu_t)=\frac{1}{2}e^{-2\alpha t}\sum_{i=1}^N\beta_i|x_0^i|^2\le\frac{1}{2}e^{-\alpha t}\mm^2(\tilde\mu_0)=e^{-\alpha t}\V(\tilde\mu_0).
\end{equation*}
Following a similar reasoning as that reported at the end of the proof of Corollary \ref{cor:Lviab}, we conclude that $\SS$, defined through $(\tilde\mu_t)_t$, is nonincreasing in $[0,+\infty)$.

\medskip

We can give an intepretation of the proposed example e.g. in social sciences. Indeed, $\F$ can be interpreted as modeling individual behavioral adjustments in a social system, where $x\in\R^d$ represents the state of an agent (e.g., opinion, wealth, or position), and $\nu\in\prob_2(\R^d)$ denotes the distribution of the overall population. The available control at each state is given by a ball centered at the origin with radius proportional to both the agent's deviation from consensus (represented by the origin) and the global dispersion of the population. The model captures situations where agents adjust their state under bounded effort and directional uncertainty, but with a responsiveness that increases with individual deviation and systemic disorder. The asymptotic reachability of the origin corresponds to a collective convergence toward consensus or centralization. 

\subsection{The squared Wasserstein distance as a Lyapunov function} This example, presented for simplicity in dimension $d=2$, is very similar to an example in \cite{aver}. 
We consider the single-valued vector field $\F$ defined, for $ x \in  \R ^2$ and $ \mu \in\prob_2(\R^d)$, as follows
$$ \F (x, \mu) : =A x + \int _ {\R ^2} B y \d \mu (y),$$ where  $A :=   \begin{pmatrix}0 &1\\ 
-1 & 0 
\end{pmatrix}$  and $B$ a negative definite matrix such that for some $k>0$  $$\langle Bz, z\rangle \leq -k |z | ^2 \quad \forall z \in \R ^2.$$ We will consider the associated continuity equation
\begin{equation} \label{CE}
\partial _t \mu _t +\mathrm{div}\left[\left(A x + \int _ {\R ^2} B y \d \mu _t(y)\right)\mu_t\right]=0 \quad\mbox{for a.e. $t$ and $\mu_t$-a.e. $x\in \mathbb R^2 $. }
\end{equation} We consider $\bar \nu $ to be the probability measure with density $\frac1c e ^{-(x_1^2 +x^2_2)}$ (where $c$ is a normalizing constant). It can be shown (cf. \cite[Section 4.1]{aver}) that $ \bar \nu $ is an equilibria of \eqref{CE}. We wish to prove that $\V:\prob_2(\R^2)\to[0,+\infty)$, defined by $\V(\nu):=\frac12 W^2_2 (\nu, \bar \nu )$ is a Lyapunov function for $\F$ with rate $ \varphi (r) := 2 k r$.

 For this aim, we claim that $\V$ satisfies \eqref{eq:Linequality}.

According to \cite[Theorem 3.23, Proposition and Definition 3.1, Proposition 3.11, Remark 3.12]{Jimenez}, it is enough to show that
\begin{equation*}
\varphi(\V(\nu))+\HF(\nu,p)\le 0,\qquad\forall\,\nu\in\prob_2(\R^d),\,\forall\,p\in\partial_{\text{AG}}^-\V(\nu),
\end{equation*}
where $\partial_{\text{AG}}^-\V(\nu)$ is defined in \cite{ag} to be the set of all $p\in\mathrm{Tan}_\nu\prob_2(\R^2)$ (cf. \cite[(8.0.2)]{ags}) such that
for all $\mu\in\prob_2(\R^2)$ and $\ssigma\in\Gamma_o(\nu,\mu)$,
\[\V(\mu)-\V(\nu)\ge \int_{\R^2\times\R^2}\langle p(x),y-x\rangle \d\ssigma(x,y)+o(W_2(\mu,\nu)).\]

 Fix $ \nu  \in\prob_2(\R^d) $; combining \cite[Section 10.2, Theorem 10.2.6]{ags}, \cite[Theorem 3.21, Corollary 3.22]{G2} and \cite[Theorem 3.2]{Al-J} it is possible to show that,
if there exists a measurable function $T:\R^d\to\R^d$ such that $\Gamma_o(\nu,\bar{\nu})=\left\{(\mathrm{id}_{\R^d},T)_\sharp\nu\right\}$, then $\partial_{\text{AG}}^-\V(\nu)=\left\{\mathrm{id}_{\R^d}-T\right\}$; otherwise, $\partial_{\text{AG}}^-\V(\nu)=\emptyset$. Continuity of $\V$ ensures the existence of a dense subset of $\prob_2(\R^d)$ where $\partial_{\text{AG}}^-\V$ is not empty (cf. e.g. \cite[Proposition 3.22]{JMQ21}).

A straightforward adaptation of the argument in \cite[Section 4, Theorem 4.4]{aver} allows us to conclude that for all $ p$ of the form $\mathrm{id}_{\R^2}-T$, with $T$ as above, we get
$$ 2 k \V( \nu) +  \HF(\nu,p) \leq 0.$$
 Our claim is proved. By Theorem \ref{thm:main}, $\bar \nu $ is asymptotically reachable for the dynamical system \eqref{CE}.
}

\appendix
\section{A priori estimates}\label{app0}
In the following, we recall standard a priori estimates (cf. e.g. \cite[Lemma 3.5]{CMattain}, \cite[Appendix A]{CLOS}) that we used in the proof of Proposition \ref{prop:decreas} and in Lemma \ref{lem:LipUg}. We stress that the existence of a representation $\eeta$ for a curve $\mu\in\adm{[a,b]}{\bar{\mu}}$, as in the following statement, is guaranteed by the Superposition Principle \cite[Theorem 8.2.1]{ags} and the growth property of $\F$ \eqref{eq:growthF}.

In the following, $L$ is the Lipschitz constant of $\F$ and $K_\F$ is  given by \eqref{eq:KF}. The map $\ev_t:\R^d\times\rmC([a,b])\to\R^d$ denotes the evaluation operator at time $t\in[a,b]$, defined by $\ev_t(x,\gamma):=\gamma(t)$.
\begin{lemma}\label{lem:attainCM}
Assume Hypothesis \ref{hypo}. Let $0\le a< b$, $\bar{\mu}\in\prob_2(\R^d)$ and $\mu\in\adm{[a,b]}{\bar{\mu}}$. Let $\eeta\in\prob(\R^d\times\rmC([a,b]))$ be such that
\begin{enumerate}
\item ${\ev_t}_\sharp\eeta=\mu_t$ for all $t\in[a,b]$;
\item for $\eeta$-a.e. $(x,\gamma)\in\R^d\times\rmC([a,b])$ and a.e. $t\in[a,b]$, $\gamma$ is an absolutely continuous curve such that $\gamma(a)=x$ and $\dot\gamma(t)\in\F(\gamma(t),\mu_t)$.
\end{enumerate}
Then
\begin{enumerate}
\item $\displaystyle\m{\mu_t}\le e^{L(b-a) e^{L(b-a)}}C_{[a,b]}(\bar{\mu})$, for any $t\in[a,b]$;
\item $\displaystyle\int_{\R^d\times\rmC([a,b])}\|\dot\gamma\|^2_{L^\infty([a,b])}\d\eeta(x,\gamma)\le D_{[a,b]}(\bar{\mu})$,
\end{enumerate}
where
\begin{equation}\label{last}
\begin{split}
C_{[a,b]}(\bar{\mu})&:=e^{L(b-a)}\left(\m{\bar{\mu}}+(K_\F (b-a)\right);\\
D_{[a,b]}(\bar{\mu})&:=\left(K_\F+L\,\left[e^{L(b-a) e^{L(b-a)}}{\color{black}C_{[a,b]}(\bar{\mu})}+{\color{black}C_{[a,b]}(\bar{\mu})}+L(b-a)e^{L(b-a)e^{L(b-a)}} {\color{black}C_{[a,b]}(\bar{\mu})}\right]\right)^2.
\end{split}
\end{equation}
\end{lemma}
\begin{proof}
For $\eeta$-a.e. $(x,\gamma)$ and a.e. $t\in[a,b]$ we have
\[\dot\gamma(t)\in \F(\gamma(t),\mu_t)\subset \F(0,\delta_0)+L\left(|\gamma(t)|+\m{\mu_t}\right) \overline{B(0,1)},\]
hence
\begin{equation}\label{eq:m1}
|\dot\gamma(t)|\le K_\F+L\left(|\gamma(t)|+\m{\mu_t}\right).
\end{equation}
Moreover, for any $t\in[a,b]$,
\[|\gamma(t)|-|\gamma(a)|\le \int_a^t|\dot\gamma(\tau)|\d\tau\le K_\F (b-a)+L\left(\int_a^t\left(|\gamma(\tau)|+\m{\mu_\tau}\right)\d\tau\right).\]
By Gr\"{o}nwall's inequality, this gives
\begin{equation}\label{eq:m2}
|\gamma(t)|\le e^{L(b-a)}\left(|\gamma(a)|+K_\F (b-a)+L\int_a^t\m{\mu_\tau}\d\tau\right),
\end{equation}
and together with \eqref{eq:m1}, we have
\begin{equation}\label{eq:m3}
\|\dot\gamma\|_{L^\infty([a,b])}\le K_\F +L\left[\sup_{t\in[a,b]}\m{\mu_t}+e^{L(b-a)}\left(|\gamma(a)|+K_\F (b-a)+L\int_a^b\m{\mu_\tau}\d\tau\right)\right].
\end{equation}
By taking the $L^2_{\eeta}$-norm of \eqref{eq:m2} and using triangular inequality, we get
\begin{equation}\label{eq:m4}
\m{\mu_t}\le e^{L(b-a)}\left(\m{\bar{\mu}}+K_\F (b-a)+L\int_a^t\m{\mu_\tau}\d\tau\right),
\end{equation}
thus, by Gr\"{o}nwall's inequality,
\[\m{\mu_t}\le e^{L(b-a)e^{L(b-a)}}\left[e^{L(b-a)}\left(\m{\bar{\mu}}+K_\F (b-a)\right)\right],\]
which proves item (1).
By taking the $L^2_{\eeta}$-norm of \eqref{eq:m3} and using triangular inequality we get item (2).
\end{proof}
\section{Proof of Comparison Principle}\label{sec:appA}
In this section, devoted to the proof of Theorem \ref{thm:comparison}, we recall some results proved in \cite{JMQ20} which have been used in the proof of Proposition \ref{prop:decreas}. We mention that the Comparison Principle in \cite[Theorem 2]{JMQ20} for \eqref{eq:HJBintro} is proved in the context of bounded and uniformly continuous viscosity solutions $U:[0,T]\times\prob_2(\R^d)\to \R$. Due to the lack of boundedness in our framework, we readapt the result in Theorem \ref{thm:comparison}.

We recall the following notation taken from \cite[Section 5.1]{JMQ20}. Given $\nu_1,\nu_2\in\prob_2(\R^d)$, we consider the unique optimal plan $\ppi\in\Gamma_o(\nu_1,\nu_2)$ satisfying
\[\ppi=\argmin\left\{\int_{\R^d}\left|x-\int_{\R^d}y\d\pi_x(y)\right|^2\d\nu_1(x)\,:\,\pi=\nu_1\otimes\pi_x\in\Gamma_o(\nu_1,\nu_2)\right\},\]
where the notation $\{\pi_x\}_x\subset\prob_2(\R^d)$ stands for the family obtained by disintegrating $\pi$ w.r.t. the projection on the first marginal $\nu_1$  (cf. \cite[Theorem 5.3.1]{ags}).

Denote by $\ppi^{-1}\in\Gamma_o(\nu_2,\nu_1)$ the inverse of $\ppi$, i.e. $\ppi^{-1}=\sfs_\sharp\ppi$ where $\sfs(x,y)=(y,x)$, and by $\{\ppi^{-1}_y\}_y\subset\prob_2(\R^d)$ the family obtained by disintegration of $\ppi^{-1}$ w.r.t. the projection on the first marginal $\nu_2$.
We define $p_{\nu_1,\nu_2}\in L^2_{\nu_1}$ and $q_{\nu_1,\nu_2}\in L^2_{\nu_2}$ by
\begin{equation}\label{def:pq}
\begin{cases}
p_{\nu_1,\nu_2}(x):=\displaystyle x-\int_{\R^d}y\d \ppi_x(y),\\
q_{\nu_1,\nu_2}(y):=\displaystyle y-\int_{\R^d}x\d \ppi^{-1}_y(x)
\end{cases}
\end{equation}

In what follows, we denote by $(X, d_X)$ and $(X\times X,d_{X\times X})$ the metric spaces given by
\begin{equation*}
\begin{split}
& X:=[0,T]\times \prob_2(\R^d),\\
& d_X\left(r_1,r_2\right):=\sqrt{(s_1-s_2)^2+W_2^2(\nu_1,\nu_2)},\\
& d_{X\times X}(z_1,z_2):=d_X(r_1,r_2)+d_X(\hat{r}_1,\hat{r}_2),
\end{split}
\end{equation*}
with $r_i:=(s_i,\nu_i)\in X$, $z_i:=(r_i,\hat{r}_i)\in X\times X$. Notice that both $(X, d_X)$ and $(X\times X,d_{X\times X})$ are complete.

\begin{proof}[Proof of Theorem \ref{thm:comparison}]
The proof follows mainly the structure of \cite[proof of Theorem 2]{JMQ20} with minor variations due to the lack of boundedness assumption on $u_i$.

Let
\[A:=\inf_{\nu\in\prob_2(\R^d)}\left\{u_2(T,\nu)-u_1(T,\nu)\right\}.\]
If $A= -\infty$, the statement is trivially verified. So let $A\neq-\infty$; since $u_1-A$ is still a subsolution of \eqref{eq:HJBcomp}, we can assume $A=0$ with no loss of generality. Let
\[-\xi:=\inf_{(t,\nu)\in[0,T]\times\prob_2(\R^d)}\left\{u_2(t,\nu)-u_1(t,\nu)\right\}.\]
Assume by contradiction that $-\xi<0$.
{\color{black}Without loss of generality, we may assume that $-\xi \neq -\infty$; otherwise, the proof proceeds analogously by replacing all occurrences of $\xi$ in what follows with an arbitrarily fixed $B > 0$.}
Fix $(t_0,\nu_0)\in X$ such that
\[u_2(t_0,\nu_0)-u_1(t_0,\nu_0)<-\frac{\xi}{2}.\]
By \eqref{item:comp1} we can assume $t_0\neq 0$; fix $\sigma>0$ such that $\frac{2\sigma}{t_0}\le\frac{\xi}{8}$, and let $R>0$.
Given $\eps, \eta>0$ we define the functional $\Phi_{\eps,\eta}: X\times X\to \R\cup\{+\infty\}$ by
\begin{equation*}
\begin{split}
\Phi_{\eps,\eta}(s,\nu_1,t,\nu_2)=&u_2(t,\nu_2)-u_1(s,\nu_1)+\frac{1}{2\eps}d^2_X\left((s,\nu_1),(t,\nu_2)\right)+\\
&-\eta s+\frac{\sigma}{s}+\frac{\sigma}{t},
\end{split}
\end{equation*}
if $st\ne0$ and $W_2(\nu_0,\nu_i)\le R$ for $i=1,2$; otherwise we set 
\begin{equation*}
\Phi_{\eps,\eta}(s,\nu_1,t,\nu_2)=+\infty.
\end{equation*}

By triangular inquality, notice that if $W_2(\nu_0,\nu_i)\le R$, then
\begin{equation}\label{barR}
\m{\nu_i}=W_2(\nu_i,\delta_0)\le R+\m{\nu_0}=:\bar{R},\quad i=1,2.
\end{equation}
Hence, if $st\ne0$ and $W_2(\nu_0,\nu_i)\le R$ for $i=1,2$, we have
\begin{equation*}
\begin{split}
u_2(t,\nu_2)-u_1(s,\nu_1)&=u_2(t,\nu_2)-u_2(t,\nu_0)+u_2(t,\nu_0)-u_1(s,\nu_0)+u_1(s,\nu_0)-u_1(s,\nu_1)\\
&\ge -L_{\bar{R},2} W_2(\nu_2,\nu_0)+u_2(t,\nu_0)-u_1(s,\nu_0)-L_{\bar{R},1} W_2(\nu_0,\nu_1)\\
&\ge -(L_{\bar{R},2}+L_{\bar{R},1}) R+m_2-M_1,
\end{split}
\end{equation*}
where we used assumption \ref{item:comp2} for the first inequality and \ref{item:comp1} for the existence of $m_i, M_i\in\R$ such that $m_i\le u_i(r,\nu_0)\le M_i$, for any $r\in[0,T]$.
Hence, $\Phi_{\eps,\eta}$ is bounded from below and lower semicontinuous in the complete metric space $(X\times X,d_{X\times X})$. Define $z_0=(t_0,\nu_0,t_0,\nu_0)$; by Ekeland Variational Principle, for any $\delta>0$ there exists $z_{\eps \eta \delta}=(s_{\eps \eta \delta}, \nu^1_{\eps \eta \delta},t_{\eps \eta \delta},\nu^2_{\eps \eta \delta})\in X\times X$ such that, for any $z=(s,\nu_1,t,\nu_2)$ we have
\begin{equation}\label{eq:Ekeland}
\begin{cases}
\Phi_{\eps,\eta}(z_{\eps \eta \delta})\le \Phi_{\eps,\eta}(z_0),\\
\Phi_{\eps,\eta}(z_{\eps \eta \delta})\le \Phi_{\eps,\eta}(z)+\delta d_{X\times X}(z,z_{\eps \eta \delta}).
\end{cases}
\end{equation}
In particular, $t_{\eps \eta \delta},s_{\eps \eta \delta}\neq 0$ and $W_2(\nu_0,\nu^i_{\eps \eta \delta})\le R$, $i=1,2$.

Let $z_1=(s_{\eps \eta \delta},\nu^1_{\eps \eta \delta},s_{\eps \eta \delta},\nu^1_{\eps \eta \delta})$ and $z_2=(t_{\eps \eta \delta},\nu^2_{\eps \eta \delta},t_{\eps \eta \delta},\nu^2_{\eps \eta \delta})$. By summing up the second inequality in \eqref{eq:Ekeland} with $z=z_1$ and with the choice $z=z_2$, we get
\begin{equation}\label{eq:chain1}
u_2(t_{\eps \eta \delta},\nu^2_{\eps \eta \delta})-u_2(s_{\eps \eta \delta},\nu^1_{\eps \eta \delta}) + u_1(t_{\eps \eta \delta},\nu^2_{\eps \eta \delta})-u_1(s_{\eps \eta \delta},\nu^1_{\eps \eta \delta})+\frac{1}{\eps}\rho^2_{\eps \eta \delta}
\le \rho_{\eps \eta \delta}(2\delta+\eta),
\end{equation}
where $\rho_{\eps \eta \delta}:=d_X\left((s_{\eps \eta \delta},\nu^1_{\eps \eta \delta}),(t_{\eps \eta \delta},\nu^2_{\eps \eta \delta})\right)$.

Recall that, thanks to assumptions \ref{item:comp1}, \ref{item:comp2} and \eqref{barR}, we have
\begin{equation}\label{eq:chain2}
\begin{split}
&u_i(t_{\eps \eta \delta},\nu^2_{\eps \eta \delta})-u_i(s_{\eps \eta \delta},\nu^1_{\eps \eta \delta})\\
&=u_i(t_{\eps \eta \delta},\nu^2_{\eps \eta \delta})-u_i(t_{\eps \eta \delta},\nu_0)+u_i(t_{\eps \eta \delta},\nu_0)-u_i(s_{\eps \eta \delta},\nu_0)+u_i(s_{\eps \eta \delta},\nu_0)-u_i(s_{\eps \eta \delta},\nu^1_{\eps \eta \delta})\\
&\ge -L_{\bar{R},i} W_2(\nu^2_{\eps \eta \delta},\nu_0)+m_i-M_i-L_{\bar{R},i} W_2(\nu_0,\nu^1_{\eps \eta \delta})\\
&\ge -2 L_{\bar{R},i} R+m_i-M_i,
\end{split}
\end{equation}
for $i=1,2$.

Fix $0<\eta<1$; we aim to prove that 
\begin{equation}\label{claim1CP}
\lim_{\eps, \delta\to 0^+}\rho_{\eps \eta \delta}=0 \quad\text{and}\quad \lim_{\eps, \delta\to 0^+}\frac{1}{\eps}\rho^2_{\eps \eta \delta}=0.
\end{equation}
We start by showing that the first limit cannot be $+\infty$. By contradiction, assume that there exist sequences $\{\eps_n\}_n,\,\{\delta_n\}_n$ with $\eps_n,\delta_n\to0^+$ such that $\lim_{n\to+\infty}\rho_{\eps_n \eta \delta_n}=+\infty$. Then, for $n$ sufficiently large we have $\eps_n,\delta_n\le \frac{1}{2}$ and, by \eqref{eq:chain1} and \eqref{eq:chain2}, recalling that $\eta<1$ we get
\[-2R(L_{\bar{R},1}+L_{\bar{R},2})+m_1+m_2-(M_1+M_2)+2\rho^2_{\eps_n \eta \delta_n}< 2 \rho_{\eps_n \eta \delta_n}.\]
This gives a contradiction, so the first limit in \eqref{claim1CP} is finite.

Assume by contradiction that there exist $\alpha>0$ and sequences $\{\eps_n\}_n,\,\{\delta_n\}_n$ with $\eps_n,\delta_n\to0^+$ such that $\lim_{n\to+\infty}\rho_{\eps_n \eta \delta_n}=2\alpha$. Then, for $n$ sufficiently large we have $\alpha<\rho_{\eps_n \eta \delta_n}<3\alpha$, so that from \eqref{eq:chain1} and \eqref{eq:chain2}, we have
\[-2R(L_{\bar{R},1}+L_{\bar{R},2})+m_1+m_2-(M_1+M_2)+\frac{\alpha^2}{\eps_n}\le 3\alpha (2\delta_n+\eta).\]
Since the right hand side is bounded, while the left hand side tends to $+\infty$, we get a contradiction. This proves the first relation in \eqref{claim1CP}.

In order to prove the second relation in \eqref{claim1CP}, notice that from \eqref{eq:chain1} we get
\begin{equation*}
\begin{split}
\frac{1}{\eps}\rho^2_{\eps \eta \delta}\le &\rho_{\eps \eta \delta} (2\delta+\eta)+|u_2(t_{\eps \eta \delta},\nu^2_{\eps \eta \delta})-u_2(s_{\eps \eta \delta},\nu^1_{\eps \eta \delta}) |+\\
&+ |u_1(t_{\eps \eta \delta},\nu^2_{\eps \eta \delta})-u_1(s_{\eps \eta \delta},\nu^1_{\eps \eta \delta})|.
\end{split}
\end{equation*}
By continuity of $u_1,u_2$, the right hand side tends to $0$ as $\eps,\delta\to0^+$, and this proves the second relation in \eqref{claim1CP}.

\medskip

The identities in \eqref{claim1CP} are used in the remaining part of the proof to get that for $\eps, \delta, \eta>0$ sufficiently small, we have $t_{\eps \eta \delta},s_{\eps \eta \delta}\not\in(0,T)$. This is done arguing by contradiction using \cite[Lemma 6]{JMQ20} together with assumption \eqref{item:comp3} in Theorem \ref{thm:comparison} and \eqref{eq:assHCP}. Finally, we can prove that $t_{\eps \eta \delta},s_{\eps \eta \delta}\neq T$ by using the definition of $\xi$, \eqref{eq:Ekeland} and assumption \eqref{item:comp2} in Theorem \ref{thm:comparison}. This part of the proof, being identical to \cite[Theorem 2, claims 2,3]{JMQ20}, is not reported here.

Since we finally get $t_{\eps \eta \delta},s_{\eps \eta \delta}\not\in[0,T]$, this implies $\xi=0$, hence the conclusion.
\end{proof}

The proof of the following result is provided in \cite[Theorem 3 - step1]{JMQ20}, with running cost $\mathcal L=0$.
\begin{lemma}\label{lem:Hok}
Assume Hypothesis \ref{hypo}. The Hamiltonian $\HF$ in \eqref{eq:H} satisfies the assumptions of Theorem \ref{thm:comparison} with $\omega_{\HF}(r,s)=Ls$, where $L$ is the Lipschitz constant of $\F$.
\end{lemma}

The following is an easy adaptation of \cite[Proposition 4]{JMQ20} and \cite[Theorem 3 - step2]{JMQ20} to the case of $\mm$-locally Lipschitz final cost $g$ in place of bounded and uniformly continuous. We report only the essential steps of the proof for the reader's convenience.
\begin{lemma}\label{lem:LipUg}
Assume Hypothesis \ref{hypo}. Let $g:\prob_2(\R^d)\to[0,+\infty)$ be a $\mm$-locally Lipschitz map and $U_g:[0,T]\times\prob_2(\R^d)\to\R$ be the value function of the associated Mayer problem, defined by
\begin{equation}\label{eq:MayerProbl}
U_g(t,\nu):=\inf\left\{g(\mu_{T})\,:\,\mu\in\adm{[t,T]}{\nu}\right\}.
\end{equation}
Then the following hold
\begin{enumerate}
\item\label{item:contUg} the map $t\mapsto U_g(t,\nu)$ is continuous in $[0,T]$ for all $\nu\in\prob_2(\R^d)$;
\item\label{item:LipmuUg} the map $\nu\mapsto U_g(t,\nu)$ is $\mm$-locally Lipschitz in $\prob_2(\R^d)$, uniformly w.r.t. $t$, i.e. for any $R>0$ there exists $L_{R,U_g}>0$ such that
\begin{equation*}
|U_g(t,\nu_1)-U_g(t,\nu_2)|\le L_{R,U_g} W_2(\nu_1,\nu_2),
\end{equation*}
for any $\nu_i\in\prob_2(\R^d)$ such that $\m{\nu_i}\le R$, $i=1,2$, for any $t\in[0,T]$.
\end{enumerate}
Moreover, $U_g$ is a viscosity solution of
\begin{equation*}
\begin{cases}
\partial_t U(t,\nu)+\HF(\nu,D_\nu U(t,\nu))=0,\\
U(T,\nu)=g(\nu)
\end{cases}
\end{equation*}
according with Definition \ref{def:viscosol}.
\end{lemma}
\begin{proof}
Let $R>0$ and denote by $L_{R,g}>0$ the local Lipschitz contant of $g$ as in \eqref{eq:lipphi}, let $L$ be the Lipschitz constant of $\F$ as in Hypothesis \ref{hypo}. We prove item \eqref{item:LipmuUg}. Let $\nu_1,\nu_2\in\prob_2(\R^d)$ be such that $\m{\nu_i}\le R$ for $i=1,2$, let $s\in[0,T]$ and $\mu=(\mu_t)_t\in\adm{[s,T]}{\nu_1}$ be an optimal trajectory for the optimization problem \eqref{eq:MayerProbl}, whose existence is granted by \cite[Corollary 2]{JMQ20}.
By the Gronwall estimate in \cite[Proposition 2]{JMQ20}, there exists an admissible trajectory $\theta=(\theta_t)_t\in\adm{[s,T]}{\nu_2}$ such that
\[W_2(\mu_t,\theta_t)\le e^{LT+Te^{LT}}W_2(\nu_1,\nu_2),\quad\forall t\in[s,T].\]
We get
\begin{equation}\label{eq:LipmuUg}
U_g(s,\nu_2)-U_g(s,\nu_1)\le g(\theta_T)-g(\mu_T)\le L_{\tilde{R},g}\,e^{LT+Te^{LT}}W_2(\nu_1,\nu_2),
\end{equation}
where
\[\tilde{R}=e^{LT e^{LT}+LT} (R+K_F\,T),\]
comes from Lemma \ref{lem:attainCM}(1).
By reversing the roles of $\nu_1$ and $\nu_2$ we get that $U_g(s,\cdot)$ is $\mm$-locally Lipschitz in $\prob_2(\R^d)$ uniformly w.r.t. $s\in[0,T]$.

We now prove item \eqref{item:contUg}. Let $\nu\in\prob_2(\R^d)$, $0\le s_1\le s_2 \le T$ and $\mu=(\mu_t)_t\in\adm{[s_1,T]}{\nu}$ be an optimal trajectory. Since gluing of admissible trajectories is still admissible and $\nu=\mu_{s_1}$, we have
\[U_g(s_1,\nu)-U_g(s_2,\nu)=U_g(s_2,\mu_{s_2})-U_g(s_2,\mu_{s_1})\le L_{R,g}\,e^{LT+Te^{LT}}W_2(\mu_{s_2},\mu_{s_1}),\]
thanks to \eqref{eq:LipmuUg}, and recalling again that, by Lemma \ref{lem:attainCM}(1), there exists $R>0$ such that
\[\sup_{t\in[0,T]}\m{\mu_t}\le R.\]
 By absolute continuity of $t\mapsto\mu_t$ and reversing the roles of $s_1,s_2$ we conclude.
\medskip

The proof of the last part of the statement is provided in \cite[Theorem 3 - step2]{JMQ20} with no variations.
\end{proof}

\end{document}